\documentclass[12pt,leqno]{amsart}
\usepackage{amsmath,amssymb,amsfonts}
\usepackage{hyperref, enumerate}

\newtheorem*{theorem}{Theorem}

\newtheorem*{proposition}{Proposition}
\newtheorem*{observation}{Observation}
\newcommand\inv{^{-1}}

\begin{document}

\title{Weakly Negative Circles Versus Best Clustering in Signed Graphs}

\author{Michael G.\ Gottstein}
\address{Dept.\ of Mathematics and Statistics, Binghamton University, Binghamton, NY 13902-6000, U.S.A.}
\email{mgottst2@binghamton.edu}
\author{Leila Parsaei-Majd}
\address{University of Potsdam, Germany}
\email{leila.parsaei84@yahoo.com}
\author{Thomas Zaslavsky}
\address{Dept.\ of Mathematics and Statistics, Binghamton University, Binghamton, NY 13902-6000, U.S.A.}
\email{zaslav@math.binghamton.edu}


\begin{abstract}
Clustering a signed graph means partitioning the vertices into sets (``clusters'') so that every positive edge, and no negative edge, is within a cluster.  Clustering is not always possible; the obstruction is circles with exactly one negative edge (``weakly negative circles'').  The correlation clustering problem is to cluster with the minimum number of edges that violate the clustering rule, called $Q$.  A lower bound is $w$, the maximum number of edge-disjoint weakly negative circles.  If every two such circles are edge disjoint, then $Q=w$.  We characterize signed graphs of this kind.
\end{abstract}

\maketitle

A signed graph $\Sigma=(\Gamma,\sigma)$ consist of a graph $\Gamma = (V,E)$ and a sign function $\sigma: E \to \{+,-\}$.  A standard question about a signed graph is whether it is clusterable, i.e., whether it has a partition of $V$ such that every positive edge is inside a part of the partition and every negative edge extends between parts.  Davis \cite{Davis} introduced clustering and proved that $\Sigma$ has a clustering if and only if no circle (cycle, polygon) has exactly one negative edge.  Much later, Bansal et al.\ \cite{Bansal} introduced correlation clustering for signed graphs that are not clusterable; this is the problem of finding an optimal clustering of $V$, which means a partition with the smallest possible number of edges that disagree with the definition of a clustering.  That minimum number was introduced in \cite{Dor}; we call it the ``inclusterability index'', $Q(\Sigma)$.  

Finding an optimal clustering is NP-hard \cite{Bansal} and therefore thought to be intrinsically difficult; no good general algorithm is known.  There has been much work on algorithms for clusterings that may approach optimality; we refer to \cite{Parsaei} for an example and citations.  We treat here a special case in which a perfect solution is possible: both $Q(\Sigma)$ and an optimal clustering are readily computable.

Our graphs may have multiple edges but not loops.  Call a circle $C$ ``all positive'' if every edge in it is positive, and ``weakly negative'' if exactly one edge in it is negative.  A clear lower bound for $Q(\Sigma)$ is $w(\Sigma)$, the maximum number of ``non-overlapping'' (i.e., pairwise edge-disjoint) weakly negative circles (since at least every such circle has to be eliminated for clustering to occur), although this bound is rarely attained.  
Parsaei-Majd \cite{Parsaei} noted this bound and proved (her Lemma 2.5) that $Q(\Sigma) = w(\Sigma)$ if each pair of weakly negative circles is edge disjoint.  We give a complete structural characterization of signed graphs of this kind.  For completeness we begin with a short proof and a converse of their result.  Let $t(\Sigma) :=$ the total number of weakly negative circles.

\begin{proposition}
The following properties of a signed graph $\Sigma$ are equivalent:
\begin{enumerate}[{\rm (i)}]
\item\label{Qtotal} $Q(\Sigma) = t(\Sigma)$.
\item\label{wtotal} $w(\Sigma) = t(\Sigma)$.
\item\label{nonoverlap} No two weakly negative circles share a common edge.
\end{enumerate}
\end{proposition}

\begin{proof}
It is clear that $w \leq Q \leq t$, so \eqref{wtotal} implies \eqref{Qtotal}, and that $w=t$ if and only if no two weakly negative circles overlap, \eqref{wtotal}$\iff$\eqref{nonoverlap}.  If two weakly negative circles do overlap, then by removing a common edge we show $Q < t$. Thus, if \eqref{Qtotal} $Q=t$, there are no such circles \eqref{nonoverlap}, hence $w=t$ \eqref{wtotal}.
\end{proof}

It is easy to make $w<Q<t$.  Consider $K_4$ with two nonadjacent edges signed negative.  The weakly negative circles are the triangles, so we have $w=1$, $Q=2$ (delete the negative edges), and $t=4$.

We digress slightly to note a general property of some weakly negative circles in any signed graph.

\begin{observation}
Suppose $W$ is a weakly negative circle in $\Sigma$ that is edge-disjoint from every other weakly negative circle.  Then $W$ has at most one vertex in common with each all-positive circle and each other weakly negative circle.
\end{observation}

\begin{proof}
Let $W$ have negative edge $e^-$ and let $C$ be another circle that is all positive or weakly negative and has two or more common vertices with $W$.  The common vertices divide $C$ into paths that are internally disjoint from $W$ (and, if $C$ is all positive, possibly also edges common to $W$ and $C$).  At least one such path $P_{vw}$ is all positive.  The union $W \cup P_{vw}$ is a theta graph consisting of two all-positive paths and one path with a single negative edge; these three paths form two weakly negative circles with the common edge $e^-$.
\end{proof}

Returning to the main topic, the description of graphs that satisfy \eqref{nonoverlap} is slightly complicated.  Let $\Sigma^+ = (V, E^+)$ be the positive subgraph of $\Sigma$, where $E^+ = \sigma\inv(+)$ is the set of positive edges.  A block of a graph or signed graph is a maximal inseparable subgraph, i.e., it has no cutpoints.  A block is ``trivial'' if it is isomorphic to $K_1$ or $K_2$.  A nontrivial block is a multiple edge or it is 2-connected, and every two edges are in a circle.  A ``cactus'' is a connected graph all of whose blocks are circles or isthmi.  A ``leaf'' is a vertex of degree 1.  
The union of all trivial blocks in a graph $\Gamma$ is a forest, which we call $F(\Gamma)$; this graph is a disjoint union of trees (which may have any positive number of vertices).

\begin{theorem}
In a signed graph $\Sigma$ every two weakly negative circles are edge disjoint if and only if $\Sigma$ has the following form:  

Each component tree $T$ of $F(\Sigma^+)$ may contain negative edges as long as $T$ with its negative edges forms a cactus.  Between components of $\Sigma^+$ there may be arbitrary negative edges.  There are no other negative edges.
\end{theorem}

For the proof we assume $\Sigma$ has only edge-disjoint weakly negative circles and prove it has the described structure.  The converse is reasonably clear and is left to the reader.  

\begin{proof}
Consider a connected component $\Psi$ of $\Sigma^+$.  Those of its edges that are isthmuses of $\Sigma^+$ are the edges in $F(\Sigma^+) \cap \Psi$; that is, $F(\Psi)$ is the disjoint union of the trees $T$ of $F(\Sigma^+)$ that are contained in $\Psi$.
The remaining edges of $\Psi$ are contained in nontrivial blocks of $\Psi$.  

Add to $\Psi$ a negative edge $vw$.  If there is more than one path in $\Psi$ from $v$ to $w$, then $vw$ creates two weakly negative circles with a common edge.  If there is only one such path, then $v$ and $w$ must belong to the same tree $T$ of $F(\Sigma^+)$.
Thus, any negative edge of $\Sigma$ with both endpoints in $\Psi$ has both endpoints in a tree $T$ of $F(\Psi)$.

Suppose there are two negative edges, $uv$ and $wx$, with endpoints in the same tree $T$.  The endpoints are joined by unique paths $T_{uv}$ and $T_{wx}$ in the tree, and both $T_{uv} \cup uv$ and $T_{wx} \cup wx$ are weakly negative circles.  Therefore, there can be no common edge between $T_{uv}$ and $T_{wx}$.  This implies that $T$ along with all negative edges having both endpoints in $T$ must be a cactus.  Since a cactus has only edge-disjoint circles, such negative edges may exist.

We have shown the form of the negative edges with both endpoints in a single component of $\Sigma^+$.  On the other hand, negative edges between components of $\Sigma^+$ can never form a weakly negative circle, so they are unrestricted.  That completes the proof.
\end{proof}

In a signed graph of this type there is a disagreeing set that consists of the negative edges inside clusters, so it is easy to find both a best clustering and a minimum disagreeing set.  (The converse is not true; a signed graph with a minimum disagreeing set that consists of the negative edges within the clusters of a best clustering may have weakly negative circles that are not edge disjoint; simple examples are in \cite[Figures 7(a, b)]{Parsaei}.)  

We do a crude estimate of the computational complexity of testing $\Sigma$ for having only edge-disjoint weakly negative circles, by testing for the structure in the Theorem.  For simplicity we assume a simple graph $\Gamma$, whose order is $n = |V(\Sigma)|$.
\begin{enumerate}[Step 1.]
\item Find the positive subgraph $\Sigma^+$.  This requires checking the signs of the edges, which takes time $O(n^2)$.
\item Find the connected components $\Psi$ of the positive subgraph.  Time $O(n^2)$.
\item\label{forest} Find the set $E_F$ of isthmi of $\Sigma^+$ and form the forest $F(\Sigma^+) = (V(\Sigma), E_F)$.  Time $O(n^2)$ \cite{Tarjan}.
\item Find the negative edges within $\Psi$ for each $\Psi$.  
Test each such negative edge to see if it is contained within a tree of $F(\Sigma^+)$.  If any is not, stop.  Time $O(n^2)$.
\item\label{cactus}  If not stopped, test each tree $T$ of $F(\Sigma^+)$ with its induced negative edges to see if it is a cactus.  If any such tree with negative edges is not a cactus, stop.  This requires time $O(n^2)$ \cite[Lemma 4]{ZZ}.
\item If (and only if) not stopped, success. $\Sigma$ has no overlapping weakly negative circles and an optimal clustering is the partition induced by the components of $\Sigma^+$.
\end{enumerate}

The overall time requirement is $O(n^2)$, which may not be optimal; we leave improvement to others.


\end{document}